\documentclass[reqno,a4paper]{amsart}
\usepackage{amssymb,amsmath,amscd,ascmac}
\usepackage{graphicx} 
\usepackage[all]{xy}
\usepackage[dvipdfm,backref,%
bookmarksnumbered=true,%
colorlinks,%
linkcolor=blue,%
citecolor=blue]{hyperref}
\newcommand{\R}{\mathbb{R}}
\newcommand{\Z}{\mathbb{Z}}

\newcommand{\id}{\mathop{\mathrm{Id}}\nolimits}
\newcommand{\co}{\colon\thinspace}
\newcommand{\conn}{\mathbin{\sharp}}

\newcommand{\ev}[1]{#1_\propto}
\renewcommand{\tilde}{\widetilde}

\theoremstyle{plain}
\newtheorem{theorem}{Theorem}[section]
\newtheorem{corollary}[theorem]{Corollary}
\newtheorem{proposition}[theorem]{Proposition}
\newtheorem{lemma}[theorem]{Lemma}
\theoremstyle{definition}
\newtheorem{definition}[theorem]{Definition}
\theoremstyle{remark}
\newtheorem{remark}[theorem]{Remark}

\makeatletter
 
 \@addtoreset{equation}{section}
\makeatother
\title{Regular-equivalence of $2$-knot diagrams and sphere eversions}
\author{Masamichi Takase}
\thanks{The first-named author has been
supported in part by Grant-in-Aid for Scientific Research (C),
(No.~22540074), Japan Society for the Promotion of Science.}
\address{Faculty of Science and Technology,
Seikei University, 3-3-1 Kichijoji-kitamachi, Musashino,
Tokyo 180-8633, Japan}
\email{mtakase@st.seikei.ac.jp}
\author{Kokoro Tanaka}
\thanks{The second-named author 
thank J. Scott Carter for helpful comments on the draft. 
He has been supported in part by Grant-in-Aid for Scientific Research (C),
(No.~26400082), Japan Society for the Promotion of Science.}
\address{Department of Mathematics, Tokyo Gakugei University, 
4-1-1 Nukuikita-machi, Koganei, Tokyo 184-8501, Japan}
\email{kotanaka@u-gakugei.ac.jp}
\dedicatory{Dedicated to Professor Yukio Matsumoto 
on the occasion of his seventieth birthday}



\date{\today}

\begin{document}\sloppy

\begin{abstract}
For each diagram $D$ of a $2$-knot, 
we provide a way to construct a new diagram $D'$ of the same knot 
such that any sequence of Roseman moves between $D$ and $D'$ 
necessarily involves branch points. 
The proof is done by developing the observation that no sphere eversion 
can be lifted to an isotopy in $4$-space. 
\end{abstract}

\maketitle 

\section{Introduction}\label{intro}
A \textit{surface-knot} (or a \textit{$\Sigma^2$-knot}) is 
a submanifold of $4$-space $\R^4$, 
homeomorphic to a closed connected oriented surface $\Sigma^2$. 
Two surface-knots are said to be \textit{equivalent} if they 
can be deformed to each other through an isotopy of $\R^4$. 
In this paper we mainly study $S^2$-knots, also called \textit{$2$-knot}s, 
and especially their diagrams. 
A \textit{diagram} of a surface-knot $F$ is the image of $F$ 
via a generic projection $\R^4 \to \R^3$, 
that is equipped with a height information to distinguish the $4$th coordinate. 
A diagram is composed of four kinds of local pictures 
shown in Figure~\ref{fig:diagram}, 
each of which depicts a neighborhood of a typical point --- 
a regular point, a \textit{double point}, 
an isolated \textit{triple point} or an isolated \textit{branch point}. 
Two surface-knot diagrams are said to be \textit{equivalent} 
if they are related by (ambient isotopies of $\R^3$ and) 
a finite sequence of seven Roseman moves, 
shown in Figure~\ref{fig:roseman}, where we omit height information for simplicity. 
According to Roseman \cite{R}, two surface-knots are equivalent if and only if 
they have equivalent diagrams. 

Since an oriented surface-knot has a trivial normal bundle, 
(after a suitable isotopy of $\R^4$) 
it has a diagram with no branch points (\cite{CS, Hirsch}). 
In \cite{Satoh}, 
such a ``branch-free'' diagram is called a 
\textit{regular} diagram, and 
two regular diagrams are said to be 
\textit{regular-equivalent} if they are 
related by a finite sequence of 
``branch-free'' Roseman moves, that is, 
the moves of type $D1$, $D2$, $T1$ and $T2$ in Figure~\ref{fig:roseman}. 
Recently, many studies of surface-knot diagrams up to 
suitably restricted local deformations 
have been made (\cite{Jablonowski, KOT, OT, Satoh}), 
since they are important in studying 
the (in)dependence among the seven Roseman moves (see e.\,g.\,\cite{Kawamura}). 
\begin{figure}[thbp]
\begin{center}
\setlength\unitlength{.85\textwidth}
\begin{picture}(1,0.27)%
\put(0,0.05){\includegraphics[width=\unitlength]{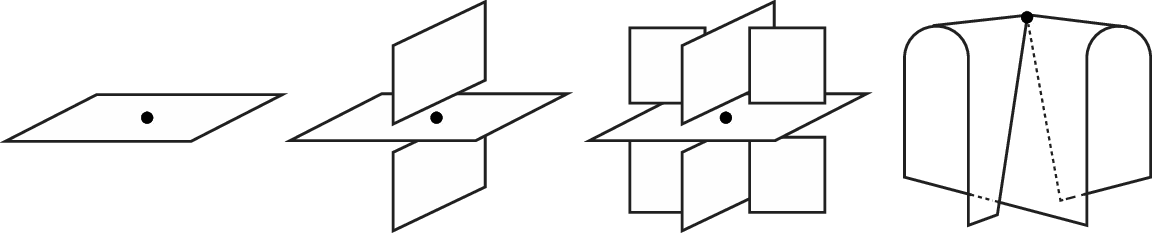}}
\put(0.0,0){\mbox{a regular point}}%
\put(0.255,0){\mbox{a double point}}%
\put(0.525,0){\mbox{a triple point}}%
\put(0.775,0){\mbox{a branch point}}%
\end{picture}%
\end{center}
\caption{Local pictures of the projection of a surface-knot}\label{fig:diagram}
\end{figure}
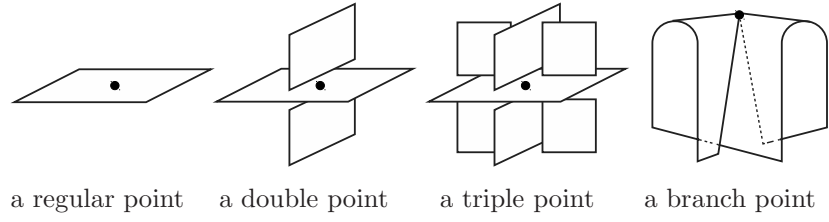

\begin{figure}[thbp]
\includegraphics[width=\textwidth]{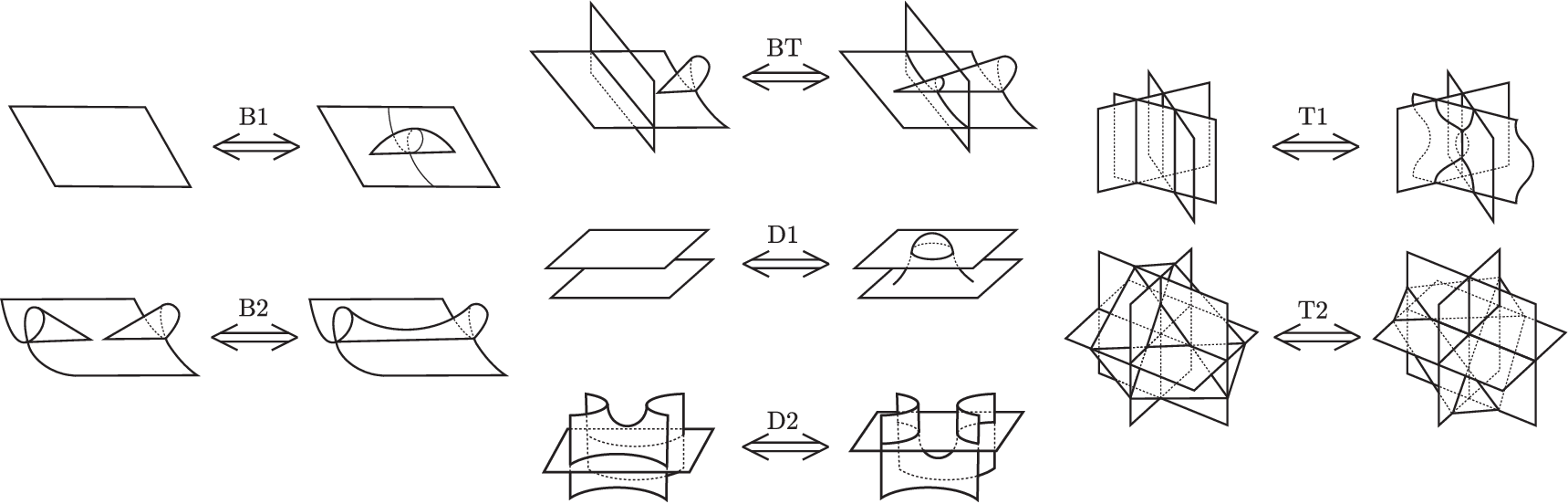}
\caption{The Roseman moves}\label{fig:roseman}
\end{figure}

In \cite{Satoh}, Satoh explicitly described a pair of regular diagrams 
of a $T^2$-knot which are mutually equivalent but not regular-equivalent, 
by computing certain elements of 
$H_1(T^2;\Z)$ derived from the double point curves of regular diagrams. 
Note that his method turns out not to be effective 
for the case of $S^2$-knots, since $H_1(S^2;\Z)=0$. 
Recently, Oshiro and the second author 
proved Satoh's result by a different method
using certain algebraic systems, called \textit{rack}s (\cite[Section 6]{OT}). 
Note that it also turns out not to be effective for the case of $S^2$-knots 
(see \cite[Section 6]{OT}). 

In this paper, we present the first example with the same property 
in the case of oriented $S^2$-knots. 
Precisely, we provide a way to construct, 
given a regular diagram $D$ of an \textit{oriented} $S^2$-knot, 
a new regular diagram $D'$ such that $D$ and $D'$ are mutually 
equivalent but not regular-equivalent (Theorem~\ref{thm:main}). 
Our method stands on basic material in immersion theory 
--- 
especially on the number of quadruple points
of a generic sphere eversion and that of generic 
immersions of $3$-manifolds in $4$-space, 
and is potentially useful for surface-knots other than $S^2$-knots. 
In the course of the argument, we observe that no sphere eversion 
can be lifted to an isotopy in $\R^4$ (Corollary~\ref{cor:eversion}). 

Throughout the paper, we work in the smooth category; all manifolds, 
immersions and embeddings 
are supposed to be differentiable of class $C^\infty$, 
unless otherwise stated. 
We always deal with 
\textit{generic immersions} and \textit{generic regular homotopies} 
defined in a natural way (see \cite{MB}).   
We also suppose that the spheres are oriented. 

We fix the following symbols. 
\begin{itemize}
\item
The orthogonal projection 
\[\begin{array}{rccc}
\pi\co&\R^{n+1}&\longrightarrow&\R^n,\\
&(x_1,x_2,\ldots,x_{n+1})&\longmapsto&(x_1,x_2,\ldots,x_n), 
\end{array}\]
defined by dropping off the last coordinate. 

\item
The reflection
\[\begin{array}{rccc}
r\co&\R^n&\longrightarrow&\R^n,\\
&(x_1,x_2,\ldots,x_n)&\longmapsto&(-x_1,x_2,\ldots,x_n), 
\end{array}\]
obtained by multiplying the first coordinate by $-1$. 

\item
The standard inclusion $j\co S^2 \hookrightarrow \R^3$. 
\end{itemize}


\section{Immersion theory}
Smale's paradox states that 
one can turn a sphere inside out. 
More precisely, 
Smale proved that the standard inclusion 
$j\co S^2\hookrightarrow\R^3$ and 
its reflection $j^*:=r\circ j$ are regularly homotopic. 
Such a regular homotopy is called \textit{an eversion}. 
Several examples of explicit eversions 
have been discovered and described 
by many mathematicians, including 
Arnold Shapiro and Bernard Morin \cite{FM}
(see \cite[\S2]{Aitchison} and \cite{Carter} for detailed histories). 

The following nature of a sphere eversion is 
well-known and plays a key role in this paper. 

\begin{proposition}[Max and Banchoff \cite{MB}, see also \cite{Hughes}]\label{eversion}
Every sphere eversion has an odd number of 
quadruple points. 
\end{proposition}

Proposition~\ref{eversion} states that 
in any regular homotopy 
between $j$ and $j^*$ an odd number of 
quadruple points appear. 
Nowik generalized this result for immersions 
of orientable surfaces as follows. 

Let $f,f'\co\Sigma^2\to\R^3$ be 
two regularly homotopic generic immersions. 
Then 
Nowik \cite{Nowik} proved that 
the number modulo $2$ of 
quadruple points occurring in generic 
a regular homotopy between $f$ and $f'$ 
does not depend on the choice of the 
regular homotopy; we denote this invariant 
by $Q(f,f')\in\Z/2\Z$. 
%

\begin{theorem}[Nowik \cite{Nowik2}]\label{nowik2}
Let $f\co\Sigma^2_g\to\R^3$ be a generic immersion 
of a closed orientable surface $\Sigma^2_g$ of genus $g$. Then, 
for an orientation-preserving diffeomorphism $\varphi$ 
\[
Q(f,f\circ\varphi)=\mathrm{rank}(\varphi_*-\id)\pmod2, 
\]
and for an orientation-reversing diffeomorphism $\varphi$
\[
Q(f,f\circ\varphi)=\mathrm{rank}(\varphi_*-\id)+g+1\pmod2, 
\]
where $\varphi_*$ is the homomorphism on $H_1(\Sigma^2_g;\Z/2\Z)$
induced by $\varphi$.

In particular when $\Sigma^2_g=S^2$, 
$Q(f,f\circ\varphi)$ equals $0$ or $1\pmod2$
depending on whether $\varphi$ is orientation-preserving
or orientation-reversing. 
\end{theorem}

\begin{remark}\label{rmk:conn}
For two generic immersions $f\co \Sigma\to \R^3$ and 
$f'\co \Sigma'\to \R^3$ of oriented surfaces, 
we can consider the \textit{oriented connected sum} of $f$ and $f'$ 
in the similar way to \cite[p.\,167]{Hughes2} 
(see also \cite[\S2]{Kervaire-60} and \cite[p.\,502]{Hughes}). 
Roughly, the connected sum is an immersion of 
$\Sigma\conn\Sigma'$ whose image is 
formed by connecting the images of $f$ and $f'$ 
by an untwisted tube in $\R^3$ with respect to the orientations. 
Note that the resultant immersion is subject to the choice of 
``the connecting tube'', but 
two immersions constructed in this way 
are mutually regularly homotopic through 
\textit{a regular homotopy with no quadruple points} since 
we can move the tube keeping it free from 
the (isolated) triple points of $f$ and $f'$. 
\end{remark}

\begin{figure}[thbp]
\begin{center}
\setlength\unitlength{.85\textwidth}
\begin{picture}(1,0.30)%
\put(0,0.06){\includegraphics[width=\unitlength]{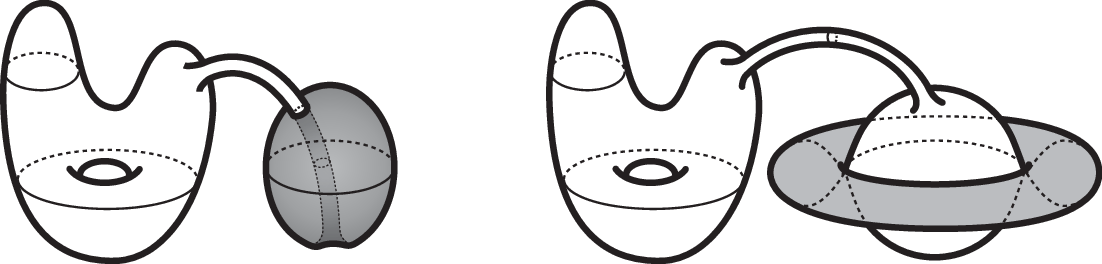}}
\put(-0.01,0){\mbox{$f\co\Sigma\to\R^3$}}%
\put(0.23,0){\mbox{$j^*\co S^2\to\R^3$}}%
\put(0.50,0){\mbox{$f\co\Sigma\to\R^3$}}%
\put(0.75,0){\mbox{$\ev{j}\co S^2\to\R^3$}}%
\put(0.18,0.31){\mbox{$f \conn j^*$}}%
\put(0.69,0.31){\mbox{$f \conn \ev{j}$}}%
\end{picture}%
\end{center}
\caption{$f\conn j^*$ and $f\conn\ev{j}$}\label{fig:conn}
\end{figure}

Now, by Nowik's result in \cite{Nowik}, 
we can easily prove the following. 

\begin{lemma}\label{nowik}
For a generic immersion $f\co\Sigma\to\R^3$ 
of a closed orientable surface $\Sigma$, 
\[
Q(f,f\conn j^*)=1\pmod2. 
\]
Note that 
$f$ and $f\conn j^*$ are regularly homotopic 
by \cite[Theorem~1]{Kaiser}. 
\end{lemma}

\begin{proof}
It is clear that 
in order to take the oriented connected sum $f\conn j^*$ 
we need to yield \textit{new} self-intersection points 
as in the left of Figure~\ref{fig:conn}, 
so that we can consider $f\conn j^*$ to be the connected sum 
$f\conn\ev{j}$ of $f$ and 
the immersion $\ev{j}\co S^2\to\R^3$ described 
in the right of Figure~\ref{fig:conn}. 

Since the immersion $\ev{j}$ is the very first (or last) step in 
an explicitly described eversion (see \cite[Fig.~2b, p.\,199]{MB}). 
Therefore, we have 
an explicit regular homotopy between the standard 
inclusion $j$ and  $\ev{j}$, and hence between 
$f=f\conn j$ and $f\conn\ev{j}$, 
with a single quadruple point (\cite[Fig.~14a, p.\,206]{MB}). 
Hence the result follows from Nowik's result \cite{Nowik}. 
\end{proof}
The following proposition, along with Lemma~\ref{nowik}, 
will play a crucial role in the proof of our main theorem 
(Theorem~\ref{thm:main}). 

\begin{proposition}\label{lift}
Let $f\co S^3\to\R^4$ be a generic immersion of the $3$-sphere. 
If there exists an embedding $\tilde{f}\co S^3\to\R^5$ 
such that $f=\pi\circ\tilde{f}$ for the orthogonal projection 
$\pi\co\R^5\to\R^4$, then 
$f$ has an even number of quadruple points. 
\end{proposition}

\begin{proof}
Let $\iota\co\R^4\to\R^5$ be the inclusion. 
Then, $\iota\circ f$ is regularly homotopic to the embedding $\tilde{f}$.
By \cite{HM}, the Smale invariant of the embedding $\tilde{f}$, 
considered as an element of $\pi_3(SO(5))$, lies 
in the kernel of the $J$-homomorphism 
$$J \co \pi_3(SO(5)) \to \pi_3^S \approx \Z/24\Z.$$ 
This implies that the cobordism class of the immersion $f$ 
is zero in the cobordism group 
of immersions of closed oriented $3$-manifolds in $\R^4$, 
which is known to be isomorphic to 
the stable $3$-stem $\pi^S_3\approx\Z/24\Z$. 
The reason is that to read off the immersion cobordism class of $f$, 
we just have to apply the Pontrjagin--Thom construction 
for $\iota\circ f(S^3)\subset\R^5$ 
and the natural normal framing 
derived from the codimension one immersion $f$ (\cite{Freedman}, see 
also \cite[p.\,180]{Hughes2}), 
and this construction exactly gives the $J$-homomorphism. 

Due to \cite{Freedman}, 
the number modulo $2$ of quadruple points of 
a generic immersion of a closed oriented $3$-manifold 
into $\R^4$ is invariant up to cobordism and 
determines the non-trivial homomorphism from 
$\pi^S_3\approx\Z/24\Z$ to $\Z/2\Z$. 
Thus, the null-cobordant immersion $f$ 
should have an even number of quadruple points. 
\end{proof}

\begin{remark}
Actually, for a generic immersion of $S^3$ into $\R^4$, 
its liftability to an embedding even in $\R^6$ implies that 
it has an even number of quadruple points 
(cf.\, \cite[Example~1]{melikhov} and \cite[Theorem~4.4]{takase}). 
\end{remark}

The following lemma is elementary. 

\begin{lemma}\label{lemma}
Let $\tilde{f}\co S^2\to\R^4$ be an embedding of 
the $2$-sphere $S^2$ 
such that 
$f:=\pi\circ\tilde{f}$ is an immersion. Then, 
there exists an embedding 
\[
\tilde{F}\co D^3\to\R^5_+=\left\{\left((x_1,x_2,x_3,x_4),t\right)\in\R^4\times[0,\infty)\right\}
\]
of the $3$-disk $D^3$ 
such that $\tilde{F}$ is an embedding extending $\tilde{f}$ and 
that $(\pi,\id)\circ\tilde{F}$ is an immersion extending $f$, where 
$(\pi,\id)$ is the projection 
\[
(\pi,\id)\co\R^5_+\to\R^4_+:\left((x_1,x_2,x_3,x_4),t\right)\mapsto\left((x_1,x_2,x_3),t\right).
\]
\end{lemma}

\begin{proof}
By Kervaire's result \cite{Kervaire}, 
there exists a $3$-disk $\tilde\Delta$ properly embedded in $\R^5_+$
such that $\partial\tilde\Delta=\tilde{f}(S^2)$. 
Then at each point $y\in\partial\tilde\Delta\subset\R^4$, 
the vector $\left(\partial/\partial x_4\right)_y$ is not 
tangent to $\partial\tilde\Delta$ 
since $f=\pi\circ\tilde{f}$ is an immersion. 
Therefore, we can assume that $\partial/\partial x_4$ 
determines a normal vector field on $\partial\tilde\Delta$ and 
extend it to a (homotopically unique) normal vector field 
on $\tilde\Delta$ in $\R^5_+$, which we denote by $\nu$. 


By the Compression Theorem \cite{RS}, 
we can isotope $(\tilde\Delta,\nu)$ with fixing it 
on $\partial\tilde\Delta$ so that 
$\nu$ becomes parallel to the $x_4$-direction (of $\R^5_+$). 
Thus, we now have the embedded $3$-disk $\tilde\Delta$ 
in $\R^5_+$ with boundary $\partial\tilde\Delta=\tilde{f}(S^2)\subset\partial\R^5_+=\R^4$, 
whose projection via $(\pi,\id)$ is an immersed $3$-disk in $\R^4_+$ 
with boundary $f(S^2)\subset\partial\R^4_+=\R^3$. 
Thus, we can have a desired embedding 
$\tilde{F}\co D^3\to\R^5_+$ 
such that $\tilde{F}$ is an embedding extending $\tilde{f}$ 
and $(\pi,\id)\circ\tilde{F}$ is an immersion extending $f$. 
\end{proof}

\section{Equivalent but not Regular-equivalent diagrams}\label{main}

For an oriented surface-knot diagram $D$, 
the \textit{mirrored} diagram of $D$ is defined to be $r(D)$ 
and the \textit{inverted} diagram of $D$ be the same diagram as $D$  
with the opposite orientation. 
We use the symbols $D^*$ for 
the mirrored diagram and $-D$ for 
the inverted diagram.

\begin{remark}
It is seen that the diagram $D$ equipped with the reversed height information 
represents the same surface-knot as the diagram $r(D)$. 
Note, however, that we always use the term \textit{mirrored diagram} 
and the symbol $D^*$ exactly for $r(D)$. 
\end{remark}

\begin{proposition}\label{prop:inv}
Let $D$ be a regular diagram of an $S^2$-knot. 
\begin{enumerate}
\item[(i)] The diagram $D$ and its mirrored diagram 
$D^*$ are not regular-equivalent. 
\item[(ii)] The diagram $D$ and its inverted diagram $-D$ are not regular-equivalent. 
\end{enumerate}
\end{proposition}

\begin{proof}
First, we prove the case (i). 
Assume that $D$ and $D^*(=r(D))$ are regular-equivalent, 
that is, 
there is a sequence of Roseman moves without branch points. 
Clearly such a sequence of Roseman moves 
gives the regular homotopy covered by 
its corresponding isotopy. Namely there exists 
an isotopy 
\[
\tilde{h}_t\co S^2\to\R^4,\ t\in[-1,1], 
\]
whose projection gives a regular homotopy 
\[
h_t (:=\pi\circ\tilde{h}_t)\co S^2\to\R^3,\ t\in[-1,1]
\]
such that $h_1(S^2)=\pi\circ\tilde{h}_1(S^2)=D$ 
and $h_{-1}(S^2)=\pi\circ\tilde{h}_{-1}(S^2)=D^*$. 

Then, we will show below that the above assumption 
leads to the following two contradictory claims regarding 
the number of quadruple points 
occurring in $h_t$: 
\begin{itemize}
\item[(a)]
$Q(h_{-1},h_1)=1\pmod2$, and 
\item[(b)]
$Q(h_{-1},h_1)=0\pmod2$. 
\end{itemize}
Roughly speaking, 
(a) will be deduced from the fact that 
$h_t$ contains an essential part of an eversion, 
and (b) from the fact that $h_t$ lifts to the isotopy $\tilde{h}_t$. 

(a)
We construct a regular homotopy $\{\phi_t\}_{t\in[-3,2]}$, 
which is an eversion, 
between $\phi_2=i$ and $\phi_{-3}=r\circ i$ as follows (see Figure~\ref{fig:homotopy}). 
For $t\in[-1,1]$, we put $\phi_t=h_t$. 
Since all immersions of $S^2$ into $\R^3$ are regularly homotopic, 
there exists a regular homotopy between 
the immersion $h_1$ and the standard inclusion $i$; 
let $\{\phi_t\}_{t\in[1,2]}$ be such a regular homotopy 
that $\phi_1=h_1$ and $\phi_2=i$. 
Now, since the images of $h_{-1}$ and $r \circ h_{1}$ coincide 
(as oriented immersed surfaces), 
there exists an 
orientation-preserving diffeomorphism $\varphi$ of $S^2$ 
such that $r\circ h_{1}=h_{-1}\circ\varphi$. 
Then, take a regular homotopy $\{\phi_t\}_{t\in[-2,-1]}$ 
such that 
$h_{-1}\circ\varphi(=\phi_{-2})$ and $h_{-1}(=\phi_{-1})$. 
Finally, for $t\in[-3,-2]$ we put $\phi_t=r\circ\phi_{-1-t}$, 
so that 
$\phi_{-3}(=r\circ\phi_2)=r\circ i$ and $\phi_{-2}(=r\circ\phi_1)=r\circ h_1$. 
Note that the number of quadruple points 
appearing in $\{\phi_t\}_{t\in[1,2]}$ and 
that appearing in $\{\phi_t\}_{t\in[-2,-3]}$ are the same.

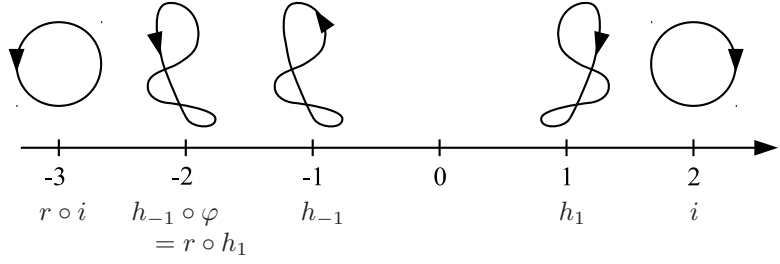
\begin{figure}[thbp]
\begin{center}
\setlength\unitlength{.8\textwidth}
\begin{picture}(1,0.33)%
\put(0,0.07){\includegraphics[width=\unitlength]{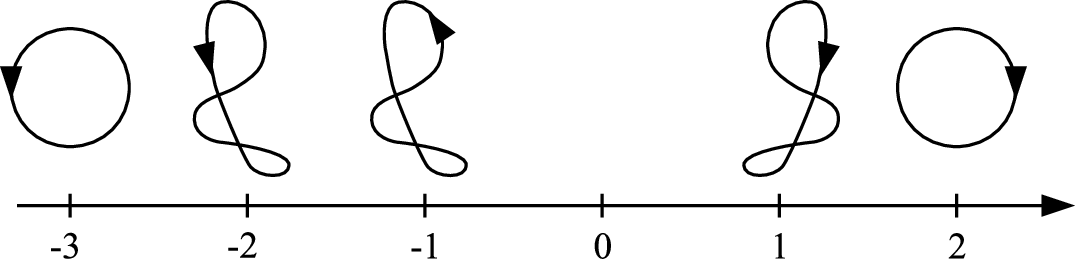}}
\put(0.04,0.04){\mbox{$r\circ i$}}%
\put(0.16,0.04){\mbox{$h_{-1}\circ\varphi$}}%
\put(0.19,0){\mbox{$=r\circ h_1$}}%
\put(0.38,0.04){\mbox{$h_{-1}$}}%
\put(0.715,0.04){\mbox{$h_{1}$}}%
\put(0.885,0.04){\mbox{$i$}}%
\end{picture}%
\end{center}
\caption{Regular homotopy $\{ \phi_t \}_{t\in[-3,2]}$}\label{fig:homotopy}
\end{figure}

Now it is clear that 
\begin{align*}
Q(\phi_{-3},\phi_{2})
&=Q(\phi_{-3},\phi_{-2})+Q(\phi_{-2},\phi_{-1})+Q(\phi_{-1},\phi_1)+
Q(\phi_1,\phi_2) \pmod2,\\
&=Q(\phi_{-2},\phi_{-1})+Q(h_{-1},h_1)+2Q(\phi_1,\phi_2)\pmod2. 
\end{align*}
Since we see that 
$Q(\phi_{-3},\phi_{2})=1\pmod2$ by Proposition~\ref{eversion} 
and $Q(\phi_{-2},\phi_{-1})=0\pmod2$
by Theorem~\ref{nowik2},
we have
\[
Q(h_{-1},h_1)=1\pmod2.
\]

(b)
By Lemma~\ref{lemma}, we can take an embedding 
\[
\tilde{F}\co D^3\to\R^5_{\ge1}=\R^4\times[1,\infty)
\]
such that $\tilde{F}$ is an embedding extending $\tilde{h}_1$ and 
that whose projection $F:=(\pi,\id)\circ\tilde{F}$ is an immersion 
extending $h_1$. 
Then, $(r,-\id)\circ\tilde{F}$, 
where 
\[\begin{array}{cccc}
(r,-\id)\co&\R^n\times\R&\longrightarrow&\R^n\times\R,\\
&(x_1,x_2,\ldots,x_n,t)&\longmapsto&(-x_1,x_2,\ldots,x_n,-t), 
\end{array}\]
gives an embedding extending $\tilde{h}_{-1}$ 
in $\R^5_{\le-1}$, whose projection $(r,-\id)\circ F$ 
is an immersion extending $h_{-1}$. 
Note that immersions $F$ and $(r,-\id)\circ F$ have 
the same number of quadruple points. 

Hence, by using $F$, $(r,-\id)\circ F$ and $\{h_t\}_{t\in[-1,1]}$, 
we can construct an immersion of the $3$-sphere 
$S^3=(-D^3)\cup (S^2\times[-1,1]) \cup D^3$ into $\R^4=\R^3\times\R$, 
which lifts to an embedding into $\R^5=\R^4\times\R$. 
Now, the number of quadruple points of this immersion, 
which is even by Proposition~\ref{lift}, 
is equal to the sum of the number of quadruple points in $\{h_t\}_{t\in[-1,1]}$ 
and twice that of $F$. 
Therefore, we have 
\[
Q(h_{-1},h_1)=0\pmod2. 
\]

This completes the proof by contradiction.

\smallskip
For the case (ii), the proof is almost the same. 
Take $h_t$ and $\tilde{h}_t$ similarly as above. 
Then, (a) follows directly from Theorem~\ref{nowik2}, 
since $h_{-1}=h_1 \circ \gamma$ for 
some orientation reversing diffeomorphism $\gamma$ of $S^2$. 
In order to show (b), 
we have only to consider 
$(\id,-\id)\circ F$ instead of $(r,-\id)\circ F$ 
in the above argument. 
\end{proof}

Let $F$ be an oriented surface-knot. 
We say that $F$ is \textit{$(+)$-amphicheiral} if 
a diagram of $F$ is equivalent to its mirrored diagram 
and that $F$ is \textit{invertible} if a diagram of $F$ is 
equivalent to its inverted diagram. 
Proposition~\ref{prop:inv} implies the following: 

\begin{theorem}\label{thm:inv}
Let $D$ be a regular diagram of an $S^2$-knot $K$. 
\begin{enumerate}
\item[(i)]
If $K$ is a $(+)$-amphicheiral knot, then the two diagrams $D$ and $D^*$ of $K$ 
are equivalent but not regular-equivalent. 
\item[(ii)] 
If $K$ is an invertible knot, then two diagrams $D$ and $-D$ of $K$ 
are equivalent but not regular-equivalent. 
\end{enumerate}
\end{theorem}

We have the following direct corollary. 

\begin{corollary}\label{cor:eversion}
The trivial diagram of the trivial $S^2$-knot 
is not regular-equivalent to itself with the reversed orientation. 
In other words, we cannot lift a sphere eversion to any isotopy in $\R^4$. 
\end{corollary}

It is natural to try to generalize Theorem~\ref{thm:inv} 
for other $S^2$-knot diagrams. 
In fact, we can prove in a similar way 
Theorem~\ref{thm:main} below, 
which enables us to construct, 
from any given $S^2$-knot diagram $D$, a new diagram 
equivalent but not regular-equivalent to $D$. 

We first need the following definition. 

\begin{definition}
\label{dfn:everted}
Let $D_0$ be a regular diagram of the trivial $S^2$-knot 
which is obtained by spinning the tangle diagram 
\begin{minipage}{12pt}
  \includegraphics[width=12pt]{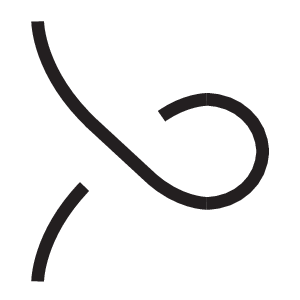}
\end{minipage}, 
and $y_0\in D_0$ be the regular point of $D_0$ as in Figure~\ref{fig:everted}. 
For a surface-knot diagram $D$, choose a regular point $y$ of $D$ and 
take an oriented connected sum of $D$ and $D_0$ by introducing 
an untwisted tube connecting (disks around) $y$ and $y_0$ 
as in Figure~\ref{fig:everted}. 
Then the resulting diagram is called an \textit{everted} diagram of $D$ 
and denoted by $\ev{D}$. 
Although $\ev{D}$ depends on the choices of $y$ and the connecting tube, 
but it is well-defined up to regular-equivalence. 
\end{definition}

\begin{figure}[thbp]
\includegraphics[width=0.70\textwidth]{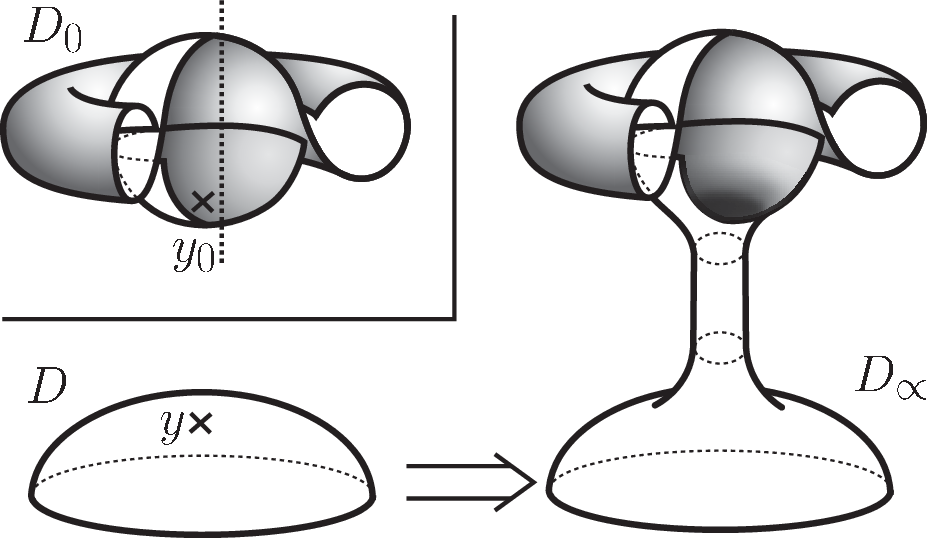}
\caption{An everted diagram of $D$}\label{fig:everted}
\end{figure}

\begin{theorem}\label{thm:main}
Let $D$ be a regular diagram of an $S^2$-knot $K$. 
Then, $D$ and its everted diagram $\ev{D}$ 
are equivalent but not regular-equivalent. 
\end{theorem}

\begin{proof}
As in the proof of Theorem~\ref{thm:inv}, 
assume that $D$ and $\ev{D}$ are regular-equivalent, 
that is, assume the existence of 
an isotopy 
\[
\tilde{h}_t\co S^2\to\R^4,\ t\in[0,1], 
\]
whose projection gives a regular homotopy 
\[
h_t\co S^2\to\R^3,\ t\in[0,1]
\]
such that $h_{0}(S^2)=\pi\circ\tilde{h}_{0}(S^2)=\ev{D}$ 
and $h_1(S^2)=\pi\circ\tilde{h}_1(S^2)=D$. 
From this assumption we will deduce 
the following two contradictory claims: 
\begin{itemize}
\item[(a)]
$Q(h_{0},h_1)=1\pmod2$, and 
\item[(b)]
$Q(h_{0},h_1)=0\pmod2$. 
\end{itemize}

(a) 
Put $j^*=r\circ j\co S^2\hookrightarrow\R^3$. 
Then, we can take the oriented 
connected sum $h_{1}\conn j^*$ 
of $h_{1}$ and $j^*$ (as in Remark~\ref{rmk:conn}), 
so that $(h_{1}\conn j^*)(S^2)=\ev{D}$. 
Hence, there exists an 
orientation-preserving diffeomorphism $\varphi$ of $S^2$ 
such that 
\[
h_{0}=(h_{1}\conn j^*)\circ\varphi. 
\]
Moreover, by Theorem~\ref{nowik2}, 
for any regular homotopy $\left\{h_t\co S^2\to\R^3\right\}_{t\in[-1,0]}$ 
such that $h_{-1}=h_{1}\conn j^*$ and 
$h_{0}=(h_{1}\conn j^*)\circ\varphi$, 
we have 
\[
Q(h_{-1},h_0)=0\pmod2. 
\]
Since the whole regular homotopy $\{h_t\}_{t\in[-1,0]\cup[0,1]}$ 
between $h_{-1}=h_{1}\conn j^*$ and $h_1$ has odd 
number of quadruple points by Lemma~\ref{nowik}, 
we have 
\[
Q(h_{0},h_1)=1\pmod2. 
\]

(b) 
As seen above, 
we can consider $\ev{D}\subset\R^3\times\{0\}$ as 
an \textit{oriented} connected sum of $D=h_1(S^2)$ and $D_0$.  
Hence, there exists a cobordism $C$ between 
$\ev{D}$ and the disjoint union $D \amalg D_0$, 
which is realized by attaching a $2$-handle 
along the ``meridian'' 
of the tube used in taking the connected sum. 
%
%
It is easy to see that the cobordism $C$ can be embedded 
in $\R^4\times[-1,0]$ whose projected image onto $\R^3\times[-1,0]$ 
is immersed and contains no quadruple points; 
assume that $C\subset\R^4\times[-1,0]$ such that 
$(\pi,\id)(C\cap(\R^4\times\{-1\}))=D \amalg D_0$ and 
$(\pi,\id)(C\cap(\R^4\times\{0\}))=\ev{D}$ 
(see Figure~\ref{fig:S3}).

We can extend the embedding $\tilde{h}_1$ to an embedding 
$\tilde{F}\co D^3\to\R^5_{\ge1}$ as in Lemma~\ref{lemma}. 
It is clear that we can deform the diagram $D_0$ 
to the trivial diagram just by a single Roseman move of type $D1$ 
(see Figure~\ref{fig:roseman}), 
that involves neither branch points nor quadruple points 
(in the projection onto $\R^4$). 
By capping off the trivial diagram with the embedded $3$-disk, 
we obtain an embedding $\tilde{G}\co D^3\to\R^5_{\le -1}$ 
such that 
$(\pi,\id)(\tilde{G}(D^3) \cap (\R^4\times\{-1\}))=D_0$ 
and that the projected image of $\tilde{G}(D^3)$ involves no quadruple points. 
Capping off the embedded cobordism $C$ with 
the embeddings $(\id,-\id)\circ\tilde{F}$ and $\tilde{G}$, 
we obtain an embedded $3$-disk $C'$ in $\R^5_{-}$ 
such that $(\pi,\id)(C' \cap(\R^4\times\{0\}))=\ev{D}$ (see Figure~\ref{fig:S3}). 
Note that the number of quadruple points of the projected image of $C'$ 
is equal to that of $\tilde{F}(D^3)$. 

Hence, by using the track of the isotopy $\{h_t\}_{t\in[0,1]}$, 
the embedding $\tilde{F}$, and the embedded $3$-disk $C'$, 
we can construct an immersion of $S^3$ into $\R^4$ 
covered by an embedding into $\R^5$ (see Figure~\ref{fig:S3}). 
\begin{figure}[thbp]
\includegraphics[width=0.83\textwidth]{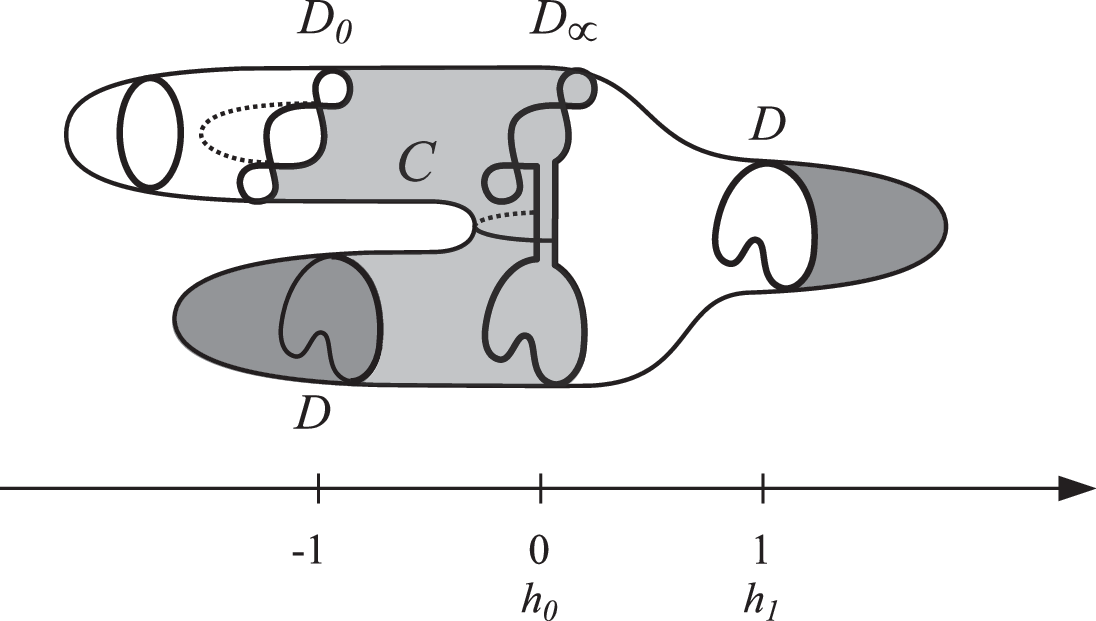}
\caption{The immersed $S^3$ in $\R^4$}\label{fig:S3}
\end{figure}

Now, the number of quadruple points of this immersion, 
which ought to be even by Proposition~\ref{lift}, 
has the same parity as 
the sum of $Q(h_0,h_1)$ and 
twice the number of quadruple points of the projected image of $\tilde{F}(D^3)$. 
Therefore we have
\[
Q(h_{0},h_1)=0\pmod2. 
\]

This completes the proof by contradiction. 
\end{proof}

\begin{remark}
For an oriented diagram $D$ of a $S^2$-knot \textit{with} branch points, 
an arbitrarily chosen regular diagram $D'$ (of the same knot) 
is not regular-equivalent to $D$, 
since it is obvious that any sequence of Roseman moves from $D'$ to $D$ 
should contain a ``branch-birth'' move B1 or B2 shown in Figure~\ref{fig:roseman}. 
\end{remark}



\begin{thebibliography}{99}
\bibitem{Aitchison}Aitchison, Iain R.: 
\textit{The Holiverse: holistic eversion of the $2$-sphere}, 
preprint 2010, available at arXiv:1008.0916. 


\bibitem{Carter}Carter, J. Scott: 
\textit{An excursion in diagrammatic algebra: Turning a sphere from red to blue}, 
Series on Knots and Everything, \textbf{48} (2012),  World Sci. Publ. 


\bibitem{CS}Carter, J. Scott; Saito, Masahico: 
\textit{Canceling branch points on projections of surfaces in $4$-space,}
Proc. Amer. Math. Soc. \textbf{116} (1992), no. 1, 229--237.


\bibitem{FM}Francis, George K.; Morin, Bernard: 
\textit{Arnold Shapiro's eversion of the sphere}, 
Math.\ Intelligencer \textbf{2} (1979/80), 200--203. 


\bibitem{Freedman}Freedman, Michael H.: 
\textit{Quadruple points of $3$-manifolds in $S^4$}, 
Comment.\ Math.\ Helv.\ \textbf{53} (1978), 385--394. 


\bibitem{Hirsch}Hirsch, M.~W.:
\textit{Immersions of manifolds}, 
Trans.\ Amer.\ Math.\ Soc.\ \textbf{93} (1959), 242--276. 


\bibitem{HM}Hughes, John F.; Melvin, Paul M.: 
\textit{The Smale invariant of a knot}, 
Comment.\ Math.\ Helv.\ \textbf{60} (1985), 615--627. 


\bibitem{Hughes}Hughes, John F.: 
\textit{Another proof that every eversion of the sphere has a quadruple point}, 
Amer.\ J.\ Math.\ \textbf{107} (1985), 501--505.


\bibitem{Hughes2}Hughes, John F.:
\textit{Bordism and regular homotopy of low-dimensional immersions},
Pacific J.\ of Math. \textbf{156} (1992), 155--184.


\bibitem{Jablonowski}Jab\l onowski, M.:
\textit{Knotted surfaces and equivalencies of their diagrams without triple points},
J.\ Knot Theory Ramifications \textbf{21} (2012), 1250019 (6 pages).


\bibitem{Kaiser}Kaiser, Uwe:
\textit{Immersions in codimension $1$ up to regular homotopy},
Arch.\ Math.\ (Basel) \textbf{51} (1988), 371--377. 


\bibitem{Kawamura}Kawamura, Kengo: 
\textit{On relationship between seven types of Roseman moves},
to appear in Topology Appl.\ 


\bibitem{KOT}Kawamura, Kengo; Oshiro, Kanako; Tanaka, Kokoro: 
\textit{Independence of Roseman moves including triple points}, 
in preparation. 


\bibitem{Kervaire-60}Kervaire, Michel A.:
\textit{Sur l'invariant de Smale d'un plongement}, (French) 
Comment.\ Math.\ Helv. \textbf{34} (1960), 127--139. 


\bibitem{Kervaire}Kervaire, Michel A.:
\textit{Knot cobordism in codimension two}, 
Manifolds-Amsterdam 1970 (Proc. Nuffic Summer School) pp.\,83--105 
Lecture Notes in Mathematics, Vol.\,197 Springer, Berlin. 


\bibitem{MB}Max, Nelson; Banchoff, Tom: 
\textit{Every sphere eversion has a quadruple point}, 
Contributions to analysis and geometry (Baltimore, Md., 1980), 
pp.\,191--209, Johns Hopkins Univ. Press, Baltimore, Md., 1981. 


\bibitem{melikhov}Melikhov, S.~A.: 
\textit{Sphere eversions and realization of mappings}, 
Proc.\ Steklov Inst.\ Math.\ \textbf{247} (2004), 143--163. 


\bibitem{Nowik}Nowik, Tahl: 
\textit{Quadruple points of regular homotopies of surfaces in $3$-manifolds}, 
Topology \textbf{39} (2000), 1069--1088. 


\bibitem{Nowik2}Nowik, Tahl: 
\textit{Automorphisms and embeddings of surfaces and quadruple points 
of regular homotopies}, 
J.\ Differential Geom.\ \textbf{58} (2001), 421--455.


\bibitem{OT}Oshiro, Kanako; Tanaka, Kokoro: 
\textit{On rack colorings for surface-knot diagrams without branch points}, 
to appear in Topology Appl.\ 


\bibitem{R}Roseman, Dennis: 
\textit{Reidemeister-type moves for surfaces in four-dimensional space}, 
Knot theory (Warsaw, 1995), 347--380, Banach Center Publ., \textbf{42}, 
Polish Acad.\ Sci., Warsaw, 1998


\bibitem{RS}Rourke, Colin; Sanderson, Brian: 
\textit{The compression theorem I}, 
Geom.\ Topol.\ \textbf{5} (2001), 399--429.


\bibitem{Satoh}Satoh, Shin: 
\textit{Double decker sets of generic surfaces in $3$-space as homology classes}, 
Illinois J.\ Math.\ \textbf{45} (2001), 823--832. 


%
%
%
%
%
%
\bibitem{takase}Takase, Masamichi: 
\textit{An Ekholm--Sz\H{u}cs-type formula for codimension one immersions 
of $3$-manifolds up to bordism}, 
Bull.\ London Math.\ Soc.\ \textbf{32} (2007), 163--178. 


\end{thebibliography}
\end{document}